\documentclass[11pt]{article}

\usepackage[utf8]{inputenc}
\usepackage[T1]{fontenc}
\usepackage[english]{babel}
\usepackage{amsmath,amssymb,amsthm,mathtools}
\usepackage{geometry}
\usepackage[expansion=false]{microtype}
\usepackage{hyperref}
\usepackage{enumitem}
\usepackage{authblk}

\hypersetup{
    colorlinks=true,
    linkcolor=black,
    citecolor=black,
    urlcolor=blue
}

\numberwithin{equation}{section}

\theoremstyle{plain}
\newtheorem{theorem}{Theorem}[section]

\newtheorem{lemma}[theorem]{Lemma}

\theoremstyle{remark}
\newtheorem{remark}[theorem]{Remark}

\newcommand{\RR}{\mathbb R}
\newcommand{\Rn}{\mathbb R^n}

\title{An averaging method for periodic solutions of
quasilinear equations in Banach spaces}

\author[1]{Jean Mawhin}
\author[2]{Jorge Novoa}

\affil[1]{Institut de Recherche en Mathématique et Physique,
Université Catholique de Louvain,
Chemin du Cyclotron 2, 1348 Louvain-la-Neuve, Belgium}

\affil[2]{Departamento de Ingeniería Matemática,
Facultad de Ciencias Físicas y Matemáticas,
Universidad de Chile,
Beauchef 851, 8370456 Santiago, Chile}

\date{}

\begin{document}

\maketitle

\begin{abstract}
The aim of the paper is to provide an extension of the averaging method to the following periodic problem for quasilinear differential equations   
$$ (\phi(t,u'))' = \varepsilon f(t,u,u',\varepsilon),\quad  u(0) = u(T), \quad u'(0) = u'(T),$$
where $u$ and $u'$ take values in a Banach space $X$ and $\phi$ is a suitable homeomorphism between a neighborhood of $0$ in $X$ and $X$. The approach follows from a given extension of the averaging method to first order systems of the form 
$$u' = g(t,v), \;  v' = h(t,u,v),\qquad  u(0) = u(T),\quad  v(0) = v(T),$$
in a product of Banach spaces obtained from a direct application of the implicit function theorem in Banach spaces to an equivalent problem. We prove that every nondegenerate zero of the corresponding averaged equation generates, for sufficiently small nonzero values of the parameter, a unique nearby periodic solution. An application is given to an infinite system of coupled forced pendulums with relativistic acceleration.

\end{abstract}

\section{Introduction}

The averaging method for the periodic problem associated to first order differential systems of the form 
\begin{eqnarray}\label{FOP}
u' = \varepsilon f(t,u,\varepsilon), \quad  u(0) = u(T),
\end{eqnarray}
with sufficiently smooth right-hand side $f :[0,T] \times \Rn \times (-\varepsilon_0,\varepsilon_0) \to \Rn$ consists in proving that
(\ref{FOP}) has a solution for $|\varepsilon|$ sufficiently small near each solution  $a^*$ of the following equation in $\Rn$
\begin{eqnarray}\label{BE}
\int_0^T f(s,a,0)\,ds = 0,
\end{eqnarray}
such that $\mbox{det} \int_0^T \partial_u f(s,a^*,0)\,ds \neq 0$. The equation (\ref{BE}) is, in general, named the {\em bifurcation equation} for problem (\ref{FOP}).

\medskip

Problems involving a small parameter arise naturally in the study of weak
perturbations of differential equations, including nonlinear oscillations
and periodically forced systems. In this setting, averaging detects
periodic responses through the simpler bifurcation equation \eqref{BE},
whose nondegenerate zeros generate periodic solutions of the perturbed
problem for sufficiently small values of the parameter.

\medskip

This result can be, and has been, proved by several classical
small-parameter methods, including the Poincar\'e operator, the
Bogoliubov--Mitropolsky averaging method, and the Cesari--Hale method.
The reader may consult \cite{RoM} for historical background and further
references.

\medskip

A quite simple and direct approach has been given in \cite{Maw1}, which consists in observing that the problem (\ref{FOP}) is trivially equivalent, for $\varepsilon \neq 0$,  to the problem 
\begin{eqnarray}\label{FON}
u' = \varepsilon f(t,u,\varepsilon), \qquad \int_0^T f(s,u(s),\varepsilon)\,ds = 0.
\end{eqnarray}
In contrast to (\ref{FOP}), which, for $\varepsilon = 0$, admits the $n$-dimensional family of solutions $u(t) = a$ with $a \in \Rn$, the problem (\ref{FON}) reduces for $\varepsilon = 0$ to $u(t) = a \in \Rn$ with $a$ solution of the bifurcation equation (\ref{BE}).

\medskip

It is therefore natural  to consider problem (\ref{FON}) as an implicit function problem  $F(u,\varepsilon) = 0$ in  the Banach space  $C^1([0,T],\Rn) \times \RR$ with $F$ defined by
 \begin{eqnarray*}
 &&F : C^1([0,T],\Rn)  \times (-\varepsilon_0,\varepsilon_0) \to C([0,T],\Rn) \times \Rn, \; \\
&& (u,\varepsilon) \mapsto \left(u' - \varepsilon f(\cdot,u(\cdot),\varepsilon), \int_0^T f(s,u(s),\varepsilon)\,ds\right).
\end{eqnarray*}
This is the approach used in \cite{Maw1} and essentially, in more general situations, in \cite{RoM}, where a direct use of the Banach fixed point theorem replaces the implicit functions argument.  This is also the approach used in \cite{CMZ} to prove an extension of the averaging method to some semilinear abstract equations in Banach spaces.

\medskip

More recently, the attention has been called to the problem of periodic solutions for $n$-dimensional systems of differential equations of the form 
\begin{eqnarray}\label{SOP}
(\phi(t,u'))' = f(t,u(t),u'(t)),
\end{eqnarray}
where, for each $t \in [0,T]$, $\phi(t,\cdot)$ is a suitable homeomorphism between some nonempty neighborhoods of $0$ in $\mathbb R^n$. The assumptions cover the classical cases of the $p$-Laplacian
($\phi(t,x)=|x|^{p-2}x$ with $p>1$), the $p(t)$-Laplacian
($\phi(t,x)=|x|^{p(t)-2}x$ with
$p\in C([0,T],(1,\infty))$), relativistic acceleration
($\phi(t,x)=\frac{x}{\sqrt{1-|x|^2}}$), and the mean curvature
operator ($\phi(t,x)=\frac{x}{\sqrt{1+|x|^2}}$), among others. The emphasis was mostly on the use of topological and variational methods, and the reader can consult the survey paper \cite{Maw2} and the recent one \cite{GMMT} for their description and the bibliography. 
Some extensions to equations of the form (\ref{SOP}) in Banach spaces have been recently given in \cite{MMN,MMN2}. Going from $\Rn$ to a Banach space not only requires imposing on $f$ some compactness properties that follow from natural ones in the $\Rn$ case, but also replacing the topological approach with some Lyapunov-Schmidt reduction combined with other nonlinear tools.

\medskip

The corresponding small parameter problem 
\begin{eqnarray}\label{SPP}
 (\phi(t,u'))' = \varepsilon f(t,u,u',\varepsilon),\quad  u(0) = u(T), \quad u'(0) = u'(T),
 \end{eqnarray}
 where results valid for small $|\varepsilon|$ only are expected,  does not seem to have been much considered so far and may deserve some consideration.

\medskip

The purpose of the present paper is to show that, under suitable
regularity and nondegeneracy assumptions on the quasilinear operator,
problem \eqref{SPP} can be treated directly by applying the
Hildebrandt--Graves extension of the classical implicit function theorem
to Banach spaces \cite{HiG} to a suitable equivalent formulation of the
problem.

\medskip

The approach is developed for problem \eqref{SPP} with
$u:[0,T]\to X$, where $X$ is a Banach space,
$\phi(t,x)=\phi(x)$, and $\phi$ is a homeomorphism from an open
neighborhood of $0$ in $X$ onto $X$, with $\phi(0)=0$,
$\phi^{-1}$ of class $C^1$, and $D\phi^{-1}(0)$ a continuous
isomorphism. This includes, in particular, the classical second-order
case $\phi(x)=x$, as well as the relativistic operator
\[
    \phi(x)=\frac{x}{\sqrt{1-|x|^2}},
    \qquad
    \phi^{-1}(y)=\frac{y}{\sqrt{1+|y|^2}},
\]
where $|x|$ denotes the norm of $x$ in $X$.

\medskip

On the other hand, our assumptions exclude important cases in which
$\phi^{-1}$ fails to be of class $C^1$ at the origin or
$D\phi^{-1}(0)$ fails to be an isomorphism. This occurs, in particular,
for the $p$-Laplacian operator
$\phi_p(x)=|x|^{p-2}x$ whenever $p\neq2$. In the subsequent paper
\cite{NoM}, we will propose an extension of the abstract averaging principle of \cite{CMZ} to some class of abstract quasilinear problems in Banach spaces, based on a nonlinear complementary correction.  This more sophisticated approach will provide similar existence conclusions in situations where the usual linearized complementary inversion is degenerate, at the expense of some compactness assumptions. 

\medskip

We now summarize the approach and the results of  the present paper. 
Since problem \eqref{SPP}, with $\phi(t,x)=\phi(x)$, can be equivalently
written as
\begin{equation}\label{HP}
    \begin{cases}
        u'=\phi^{-1}(v),\\
        v'=\varepsilon f(t,u,\phi^{-1}(v),\varepsilon),\\
        u(0)=u(T),\qquad v(0)=v(T),
    \end{cases}
\end{equation}
we first study the more general auxiliary problem in a product of Banach
spaces $X\times Y$,
\begin{equation}\label{AP}
    \begin{cases}
        u'=g(t,v),\\
        v'=\varepsilon h(t,u,v,\varepsilon),\\
        u(0)=u(T),\qquad v(0)=v(T).
    \end{cases}
\end{equation}
Following the strategy used for \eqref{FOP}, we reformulate the periodic
problem by adjoining the corresponding integral conditions and apply the
implicit function theorem at $\varepsilon=0$. Under suitable assumptions
on $g$ and $h$, Theorem \ref{EX1} yields, for every sufficiently small nonzero $\varepsilon$, a unique periodic solution near $(a^*,0)$ whenever
\[
    \int_0^T h(s,a^*,0,0)\,ds=0
\]
and the corresponding linearized averaged operator is invertible.

\medskip

Applying this result to $g=\phi^{-1}$ gives the averaging principle for
\eqref{SPP}. More precisely, Theorem \ref{EX2} shows that every
$a^*\in X$ satisfying
\[
    \int_0^T f(s,a^*,0,0)\,ds=0,
    \qquad
    \int_0^T \partial_u f(s,a^*,0,0)\,ds
    \quad\text{invertible},
\]
generates, for sufficiently small $|\varepsilon|$, a unique nearby
periodic solution converging to $a^*$ as $\varepsilon\to0$.

\medskip

Finally, we apply the result to an infinite system of differential
equations with relativistic acceleration and pendulum restoring force in
$\ell^2$. Under the condition $|\overline e_j|<1$ for every
$j\in\mathbb N$, we construct infinitely many nondegenerate solutions of
the corresponding averaged equation and obtain, for sufficiently small
nonzero $|\varepsilon|$, infinitely many geometrically distinct periodic
solutions. Other results for infinite systems of pendulum-type differential
equations with classical acceleration have been obtained by variational or
topological methods in \cite{BFG,FKS,FMW}.

\section{An averaging method for some semilinear systems}

We consider a differential system  in a product of Banach spaces $X \times Y$ of the type 
\begin{equation}\label{SS}
    \begin{cases}
        u' = g(t,v),\\
        v' = \varepsilon h(t,u,v,\varepsilon)
    \end{cases}
\end{equation}
where $g : [0,T] \times Y \to X$ and $h : [0,T] \times X \times Y \times (-\varepsilon_0,\varepsilon_0)  \to Y$ are continuous,
together with the T-periodic boundary conditions
\begin{eqnarray}\label{BC}
u(0) = u(T), \;\; v(0) = v(T).
\end{eqnarray}

\medskip

We denote by $|c|_X$ (resp. $|c|_Y$) the norm of $c \in X$ (resp. $c \in Y$),  by $C_X$ (resp. $C_Y$) the Banach space 
$C([0,T],X)$ (resp. $C([0,T],Y)$) of continuous mappings $u : [0,T] \to X$ (resp. $v : [0,T] \to Y$) with the usual uniform norm 
$\|u\|_\infty = \sup_{t \in [0,T]}|u(t)|_X$ (resp. $\|v\|_\infty = \sup_{t \in [0,T]}|v(t)|_Y$), where we omit the suffix $X$ or $Y$ to simplify the notations. We denote by $C_X^1$ (resp. $C_Y^1$) the Banach space $C^1([0,T],X)$ (resp. $C^1([0,T],Y)$) of mappings $u : [0,T] \to X$ (resp. $v : [0,T] \to Y$) of class $C^1$, with the usual norm obtained by summing the uniform norms of $u$ and $u'$ (resp. $v$ and $v'$). 

\medskip

Integrating both members of (\ref{SS}) over $[0,T]$ shows that, for $\varepsilon \neq 0$ finding a solution of (\ref{SS})-(\ref{BC}) is equivalent to finding functions $u \in C^1_X$ and $v \in C^1_Y$ verifying the nonlinear system
 \begin{align}\label{NS}
     u' - g(t,v) & = 0, & v' - \varepsilon h(t,u,v,\varepsilon) &= 0,\\\notag
     \int_0^T g(s,v(s))\,ds & = 0,&  \int_0^T h(s,u(s),v(s),\varepsilon)\,ds &= 0.
 \end{align}

Let us define the mapping
\begin{align}\label{F}
F : C^1_X \times C^1_Y \times (-\varepsilon_0,\varepsilon_0)
  &\longrightarrow C_X \times C_Y \times X \times Y, \\
(u,v,\varepsilon) &\longmapsto
  \Bigl(\, u' - g(\cdot,v),\;\;
            v' - \varepsilon h(\cdot,u,v,\varepsilon), \nonumber \\
  &\hphantom{{}\longmapsto\Bigl(\,}
            \textstyle\int_0^T g(s,v(s))\,ds,\;\;
            \int_0^T h(s,u(s),v(s),\varepsilon)\,ds \,\Bigr). \nonumber
\end{align}

It is clear that $(u,v)$ is a solution of (\ref{NS}), and hence of (\ref{SS})-(\ref{BC}) for $\varepsilon \neq 0$  if and only if $(u,v)$ is a zero of $F$. 

\medskip

In order to apply the implicit function theorem in Banach spaces to find
zeros of $F$ near some zero of $F(\cdot,\cdot,0)$, we first impose the
following assumptions on $g$:
\begin{itemize}
\item[(G1)] $g : [0,T] \times Y \to X$ and
$\partial_v g : [0,T] \times Y \to {\mathcal L}(Y,X)$ are continuous;

\item[(G2)] $g(t,0)=0$ for all $t\in[0,T]$;

\item[(G3)] For any $b\in Y$,
\[
    \int_0^T g(s,b)\,ds=0
    \quad\Longrightarrow\quad
    b=0;
\]

\item[(G4)] the operator
\[
    \int_0^T \partial_v g(s,0)\,ds
\]
has a continuous inverse.
\end{itemize}

We shall also require the following assumptions on $h$:
\begin{itemize}
\item[(H1)] $h : [0,T] \times X \times Y \times
(-\varepsilon_0,\varepsilon_0) \to Y$, $\partial_u h :
    [0,T] \times X \times Y \times
    (-\varepsilon_0,\varepsilon_0)
    \to \mathcal L(X,Y),$
and $\partial_v h :
    [0,T] \times X \times Y \times
    (-\varepsilon_0,\varepsilon_0)
    \to \mathcal L(Y,Y)$
are continuous;

\item[(H2)] there exists $a^*\in X$ such that
\begin{itemize}
\item[(i)]
\[
    \int_0^T h(s,a^*,0,0)\,ds=0;
\]

\item[(ii)] the operator
\[
    \int_0^T \partial_u h(s,a^*,0,0)\,ds
\]
has a continuous inverse.
\end{itemize}
\end{itemize}

\begin{lemma}\label{eps0}
If assumptions  (G1)--(G3) hold, then $F(u,v,0) = 0$ if and only if $v \equiv 0$, $u(t) = a \in X$ for all $t \in [0,T]$ and $\int_0^T h(s,a,0,0)\,ds = 0$.
\end{lemma}
\begin{proof}
The sufficiency is immediate. Conversely, if $F(u,v,0) = 0$, then $v' = 0$ so that $v(t) = b \in Y$ for all $t \in [0,T]$,  Hence $\int_0^T g(s,b)\,ds = 0$, so that (G3) implies that $b = 0$. Consequently, using (G2), $u' = 0$, so $u(t) = a \in X$ for all $t \in [0,T]$ and $\int_0^T h(s,a,0,0)\,ds = 0$.
\end{proof}

\medskip

Standard computations imply that, for $u,w\in C_X^1$ and
$v,z\in C_Y^1$, we have
\begin{align}
\partial_{u,v}F(u,v,\varepsilon)[w,z]
=\Bigg(&
w'-\partial_v g(\cdot,v)z,\;
z'-\varepsilon\partial_u h(\cdot,u,v,\varepsilon)w
   -\varepsilon\partial_v h(\cdot,u,v,\varepsilon)z,
\label{DF}
\\
&\int_0^T \partial_v g(s,v)z(s)\,ds,\;
\int_0^T
\Big[
\partial_u h(s,u,v,\varepsilon)w(s)
+\partial_v h(s,u,v,\varepsilon)z(s)
\Big]\,ds
\Bigg).
\notag
\end{align}

Consequently, for each $a\in X$,
\begin{align}
\partial_{u,v}F(a,0,0)[w,z]
=\Bigg(&
w'-\partial_v g(\cdot,0)z,\;
z',\;
\int_0^T \partial_v g(s,0)z(s)\,ds,
\\
&\int_0^T
\Big[
\partial_u h(s,a,0,0)w(s)
+\partial_v h(s,a,0,0)z(s)
\Big]\,ds
\Bigg).
\notag
\end{align}

The previous computation identifies the linearization of $F$ at
$(a^*,0,0)$. The next lemma shows that, under the assumptions introduced
above, this linear operator is invertible, which is necessary for the
application of the implicit function theorem.

\begin{lemma}\label{TK}
If assumptions (G1)-(G4) and (H1) - (H2) hold, then   $\partial_{u,v}F(a^*,0,0) : C^1_X \times C^1_Y \to C_X \times C_Y \times X \times Y$ is a continuous isomorphism.
\end{lemma}
\begin{proof}
First, we can directly show the existence of positive constants $K_j$ (with $1 \leq j \leq 4$) such that,
\begin{eqnarray*}
\| \partial_{u,v}F(a^*,0,0)[w,z]\|_\infty \leq K_1\|w\|_\infty + K_2\|w'\|_\infty + K_3\|z\|_\infty + K_4\|z'\|_\infty
\end{eqnarray*}
for all $w \in C^1$ and $z \in C^1$, and hence $\partial_{u,v} F(a^*,0,0)$ is continuous.

\medskip

We now show that $\partial_{u,v}F(a^*,0,0)$ is invertible.
Given any $p\in C_X$, $q\in C_Y$, $A\in X$, and $B\in Y$, consider
the equation
\[
    \partial_{u,v}F(a^*,0,0)[w,z]=(p,q,A,B),
\]
which, by definition, is equivalent to
\begin{equation}
\begin{cases}
& w'(t)-\partial_v g(t,0)z(t) = p(t),\\
& z'(t) = q(t),\\
&\int_0^T \partial_v g(\sigma,0)z(\sigma)\,d\sigma = A,\\
&\int_0^T
\Big[
\partial_u h(\sigma,a^*,0,0)w(\sigma)
+\partial_v h(\sigma,a^*,0,0)z(\sigma)
\Big]\,d\sigma = B.
\end{cases}
\end{equation}
by setting
\begin{equation*}
\begin{aligned}
m(t) &:= \partial_v g(t,0),
&\qquad
M(t) &:= \int_0^t m(\sigma)\,d\sigma,
&\qquad
Q(t) &:= \int_0^t q(\sigma)\,d\sigma,
\\
r(t) &:= \partial_u h(t,a^*,0,0),
&\qquad
R(t) &:= \int_0^t r(\sigma)\,d\sigma,
&\qquad
s(t) &:= \partial_v h(t,a^*,0,0).
\end{aligned}
\end{equation*}
the problem reduces to

\begin{equation}\label{forced}
\begin{cases}
& w'(t)-m(t)z(t) = p(t),\\
& z'(t) = q(t),\\
&\int_0^T m(\sigma)z(\sigma)\,d\sigma = A,\\
&\int_0^T
\Big[
r(\sigma)w(\sigma)
+s(\sigma)z(\sigma)
\Big]\,d\sigma = B.
\end{cases}
\end{equation}

with this, it follows from the second equation in \eqref{forced} that 
$$z(t) = c + Q(t) \mbox{ with } c \in Y,\quad \text{for all } t \in [0,T].$$
Replacing this identity on the third equation in \eqref{forced} gives 
$$M(T)c = - \int_0^Tm(\sigma)Q(\sigma)\,d\sigma + A,$$
and hence, using assumption (G4),
\begin{eqnarray} \label{z}
z(t) = -M(T)^{-1} \left(\int_0^Tm(\sigma)Q(\sigma)\,d\sigma - A \right) + Q(t).
\end{eqnarray}
 Now, integrating from $0$ to $t$ in the first equation  in \eqref{forced}  gives
 \begin{eqnarray}\label{w} 
w(t) = b  +  \int_0^t [m(\sigma)z(\sigma) + p(\sigma)]\,d\sigma.
\end{eqnarray}
with $b \in X$ some constant. Consequently, substituting \eqref{w} into the fourth equation of
\eqref{forced} gives
\[
R(T)b
=
B-\int_0^T
\left[
    r(\sigma)\int_0^\sigma
    \bigl(m(\tau)z(\tau)+p(\tau)\bigr)\,d\tau
    +s(\sigma)z(\sigma)
\right]\,d\sigma.
\]
Hence, by (H2)(ii),
\begin{align*}
w(t)
={}&
R(T)^{-1}\Bigg[
    B-\int_0^T
    \left(
        r(\sigma)\int_0^\sigma
        [m(\tau)z(\tau)+p(\tau)]\,d\tau
        +s(\sigma)z(\sigma)
    \right)d\sigma
\Bigg]
\\
&+\int_0^t[m(\sigma)z(\sigma)+p(\sigma)]\,d\sigma,
\end{align*}
with $z(t)$ given by (\ref{z}). Consequently,  $\partial_{u,v}F(a^*,0,0):
C_X^1\times C_Y^1 \longrightarrow C_X\times C_Y\times X\times Y$ is a bijection. Furthermore, it is easy to deduce from formulas (\ref{z}) and (\ref{w}) that 
\begin{eqnarray*}
\|w\|_\infty + \|w'\|_\infty  + \|z\|_\infty + \|z'\|_\infty \leq L_1\|p\|_\infty + L_2 \|q\|_\infty + L_3|A| + L_4|B|
\end{eqnarray*}
for some positive constants $L_j$ (with $1 \leq j \leq 4$), which implies that $\partial_{u,v}F(a^*,0,0)$ is a continuous isomorphism.
\end{proof}

\medskip

Lemmas \ref{eps0} and \ref{TK} show that $(a^*,0)$ is a zero of the
augmented problem at $\varepsilon=0$ and that the corresponding
linearization with respect to $(u,v)$ is a continuous isomorphism.
Therefore, the Hildebrandt--Graves implicit function theorem
\cite{HiG} (see also \cite{Die,Zei}) yields, for every sufficiently
small $\varepsilon$, a unique solution near $(a^*,0)$ converging to
$(a^*,0)$ as $\varepsilon\to0$.

\begin{theorem}\label{EX1}
Assume that {\rm (G1)--(G4)} and {\rm (H1)--(H2)} hold. Then there exists
$\varepsilon^*\in(0,\varepsilon_0)$ such that, for
$0<|\varepsilon|<\varepsilon^*$, problem \eqref{SS}--\eqref{BC}
has a unique solution
$(u(\cdot,\varepsilon),v(\cdot,\varepsilon))$ in a sufficiently small
neighborhood of $(a^*,0)$, and moreover
\[
    \lim_{\varepsilon\to0}
    (u(\cdot,\varepsilon),v(\cdot,\varepsilon))
    =(a^*,0)
    \qquad\text{in } C_X^1\times C_Y^1.
\]
\end{theorem}

\section{An averaging method for some second-order quasilinear systems}

Now let $X$ be a Banach space, let $\Omega\subset X$ be a nonempty open
neighborhood of $0$, let $\phi:\Omega\to X$ be a homeomorphism such that
$\phi(0)=0$, and let $f:[0,T]\times X\times X\times(-\varepsilon_0,\varepsilon_0)\to X$
be continuous.

\medskip

We consider the following periodic quasilinear boundary value problem 
\begin{eqnarray}\label{QP}
(\phi(u'))' = \varepsilon f(t,u,u',\varepsilon), \quad u(0) = u(T),\quad  u'(0) = u'(T).
\end{eqnarray}
A solution of \eqref{QP} is a function $u\in C_X^1$ such that
$u'([0,T])\subset\Omega$, $\phi\circ u'\in C_X^1$, and
\eqref{QP} is satisfied.

\medskip

If we define
\[
g:X\to X,
\qquad
h:[0,T]\times X\times X\times(-\varepsilon_0,\varepsilon_0)\to X
\]
by
\[
g(v):=\phi^{-1}(v),
\qquad
h(t,u,v,\varepsilon):=f(t,u,g(v),\varepsilon),
\]
then, by setting $\phi(u')=v$, the quasilinear problem \eqref{QP}
is equivalent to the semilinear problem
\begin{equation}\label{SP}
    \begin{cases}
        u' = g(v),\\
        v' = \varepsilon h(t,u,v,\varepsilon),\\
        u(0) = u(T),\quad v(0) = v(T),
    \end{cases}
\end{equation}
which is a special case of \eqref{SS}. 

\medskip

The application of Theorem \ref{EX1} to \eqref{SP} provides the following result.

\begin{theorem}\label{EX2}
If the following conditions are satisfied
\begin{itemize}
\item[{\it($\Phi$1)}]
$\phi^{-1}:X\to X$ is of class $C^1$ and
$D\phi^{-1}(0):X\to X$ is a continuous isomorphism;
\item[(F1)] $\partial_u f, \; \partial_v f : [0,T] \times X \times X \times (-\varepsilon_0,\varepsilon_0) \to {\mathcal L}(X,X)$ are continuous;
\item[(F2)] there exists $a^* \in X$ such that 
\begin{itemize}
\item[(i)]$\int_0^T f(s,a^*,0,0)\,ds = 0$
\item[(ii)] $\int_0^T \partial_u f(s,a^*,0,0)\,ds$ has a continuous inverse.
\end{itemize}
\end{itemize}
Then there exists $\varepsilon^*\in(0,\varepsilon_0)$ such that, for each
$0<|\varepsilon|<\varepsilon^*$, problem \eqref{QP} has a unique solution
$u(\cdot,\varepsilon)$ in a sufficiently small neighborhood of $a^*$,
and
\[
    u(\cdot,\varepsilon)\longrightarrow a^*
    \qquad\text{in } C_X^1
\]
as $\varepsilon\to0$.
\end{theorem}
\begin{proof}
We apply Theorem \ref{EX1} to the equivalent problem \eqref{SP}.
Since $g=\phi^{-1}$ and $\phi(0)=0$, assumptions (G1)--(G3)
follow from ($\Phi$1). Moreover,
\[
    \int_0^T Dg(0)\,ds
    =T\,D\phi^{-1}(0),
\]
so (G4) also follows from ($\Phi$1). Assumption (H1) follows
from (F1) and the chain rule, while (H2) follows directly from
(F2).
\end{proof}



\section{A simple example on infinitely many coupled forced relativistic pendulums}

Let $\ell^2$ be the classical Hilbert space of real sequences
$x=(x_j)_{j\in\mathbb N}$ such that $\sum_{j=1}^\infty x_j^2$ converges,
with inner product
$\langle x,y\rangle=\sum_{j=1}^\infty x_jy_j$ and corresponding norm
$|x|=\langle x,x\rangle^{1/2}$.

\medskip

Let $(c_{jk})_{j,k\in\mathbb N}$ be an infinite matrix such that
$\sum_{j=1}^\infty\sum_{k=1}^\infty c_{jk}^2<+\infty$, and let
$C:\ell^2\to\ell^2$ be the continuous linear mapping defined by
\[
(Cx)_j=\sum_{k=1}^\infty c_{jk}x_k
\qquad (j\in\mathbb N).
\]

For $e:[0,T]\to\ell^2$ continuous, let us consider the infinite
relativistic-type system
\begin{equation}\label{PS}
\begin{cases}
\displaystyle
\left(\frac{u'_j}{\sqrt{1-|u'|^2}}\right)'
+\varepsilon\sum_{k=1}^\infty c_{jk}u'_k
+\varepsilon\sin u_j
=\varepsilon e_j(t),\\
u_j(0)=u_j(T),\qquad u'_j(0)=u'_j(T)
\qquad (j\in\mathbb N).
\end{cases}
\end{equation}
The coupling comes from the differential operator and from the term $Cu'$.
If $u$ is a solution of \eqref{PS}, adding an integer multiple of $2\pi$
to any finite number of its components provides another solution.
Two solutions $u$ and $v$ of \eqref{PS} are called geometrically distinct
if they cannot be obtained from one another in this way.

\medskip

We seek solutions $u\in C^1_{\ell^2}$ such that $|u'|<1$ and satisfying
\eqref{PS}. Notice that if \eqref{PS} has a solution, then, by integrating
over $[0,T]$, we obtain, for $\varepsilon\neq0$,
\[
\int_0^T\sin u_j(s)\,ds
=
\int_0^T e_j(s)\,ds
\qquad (j\in\mathbb N).
\]
Hence, setting
\[
\overline e_j:=\frac1T\int_0^T e_j(s)\,ds,
\]
a necessary condition for the solvability of \eqref{PS} is
\[
|\overline e_j|\le1
\qquad (j\in\mathbb N).
\]

\medskip

The problem is a special case of \eqref{QP} with
\[
X=\ell^2,\qquad
\Omega=\{x\in\ell^2:|x|<1\},
\qquad
\phi(x)=\frac{x}{\sqrt{1-|x|^2}},
\]
and
\[
f(t,u,v,\varepsilon)
=
\left(
e_j(t)-\sin u_j-\sum_{k=1}^\infty c_{jk}v_k
\right)_{j\in\mathbb N}.
\]
Moreover,
\[
\phi^{-1}(y)=\frac{y}{\sqrt{1+|y|^2}}
\]
is of class $C^1$ on $\ell^2$ and satisfies
$D\phi^{-1}(0)=I_{\ell^2}$, so condition $(\Phi1)$ holds.
\medskip

Because $|\sin x-\sin y|\le|x-y|$ for all $x,y\in\RR$, it is easy to
see that $f$ satisfies the continuity requirements. Furthermore, for each
$w\in\ell^2$,
\[
(\partial_u f(t,u,v)-\partial_u f(t,x,z))w
=
-\bigl((\cos u_j-\cos x_j)w_j\bigr)_{j\in\mathbb N},
\]
and hence
\[
|[\partial_u f(t,u,v)-\partial_u f(t,x,z)]w|^2
\le |u-x|^2|w|^2.
\]
Similarly,
\[
\partial_v f(t,u,v)w=-Cw,
\]
and hence assumption {\rm (F1)} holds.

\medskip

Now, for $a=(a_j)_{j\in\mathbb N}$,
\[
\int_0^T[e_j(t)-\sin a_j]\,dt
=
T[\overline e_j-\sin a_j].
\]
Assume that $|\overline e_j|<1$ for every $j\in\mathbb N$. Since
$\overline e\in\ell^2$, we have $\overline e_j\to0$, and hence
\[
\rho:=\sup_{j\in\mathbb N}|\overline e_j|<1.
\]
The sequence
$a^{(0)}=(\arcsin\overline e_j)_{j\in\mathbb N}$ belongs to $\ell^2$.
To obtain further zeros in $\ell^2$, the second branch
$\pi-\arcsin\overline e_j$ may be chosen only for finitely many indices;
indeed, choosing it infinitely often would produce a subsequence converging
to $\pi$, and hence a sequence which does not belong to $\ell^2$.
\medskip

For each $m\in\mathbb N$, let $a^{(m)}$ be obtained from $a^{(0)}$ by
replacing only its $m$th component by
$\pi-\arcsin\overline e_m$. Then every $a^{(m)}$ belongs to $\ell^2$ and
solves the bifurcation equation, and the sequences  $a^{(m)}$ are geometrically
distinct. Moreover,
\[
\inf_{m\in\mathbb N}\inf_{j\in\mathbb N}
|\cos a_j^{(m)}|
\ge\sqrt{1-\rho^2}>0,
\]
so the corresponding linearized bifurcation operators have continuous
inverses, uniformly bounded with respect to $m$.
\medskip

Since the Lipschitz estimates involved in the implicit function theorem are
also independent of $m$, Theorem~\ref{EX2} yields an
$\varepsilon^*>0$ such that, for every
$0<|\varepsilon|<\varepsilon^*$, problem \eqref{PS} has a periodic
solution close to each $a^{(m)}$. Therefore \eqref{PS} has infinitely many
geometrically distinct periodic solutions.

\begin{remark}\label{findim}
When $\ell^2$ is replaced by $\Rn$ in \eqref{PS}, the above argument gives,
for sufficiently small nonzero $|\varepsilon|$, at least $2^n$
geometrically distinct solutions, one near each
$a^*=(a_1^*,\ldots,a_n^*)$ with
\[
a_j^*\in\{\arcsin\overline e_j,\pi-\arcsin\overline e_j\},
\qquad 1\le j\le n.
\]
This result is optimal for sufficiently small $|\varepsilon|$. Indeed, if
$C=0$ and $e\equiv0$, the problem becomes
\begin{equation}\label{PS0}
\begin{cases}
\displaystyle
\left(\frac{u'_j}{\sqrt{1-|u'|^2}}\right)'
+\varepsilon\sin u_j=0,
\qquad 1\le j\le n,\\
u(0)=u(T),\qquad u'(0)=u'(T).
\end{cases}
\end{equation}
Let $\overline u=T^{-1}\int_0^T u(t)\,dt$. Testing \eqref{PS0} with
$u-\overline u$, using the periodicity of $u$ and Wirtinger's inequality, we get,
\[
\int_0^T |u'|^2\,dt
\le
\int_0^T\frac{|u'|^2}{\sqrt{1-|u'|^2}}\,dt
\le
|\varepsilon|
\left(\frac{T}{2\pi}\right)^2
\int_0^T |u'|^2\,dt.
\]
Hence $u'\equiv0$ whenever
$|\varepsilon|(T/2\pi)^2<1$ (notice that this is true for any $\varepsilon$ sufficiently small). Therefore $u(t)=a\in\Rn$ and
$\sin a_j=0$ for every $j$, so, modulo multiples of $2\pi$,
$a_j\in\{0,\pi\}$. Thus \eqref{PS0} has exactly $2^n$
geometrically distinct periodic solutions.
\end{remark}

\section*{Funding}
This research did not receive any specific grant from funding agencies in the public, commercial, or not-for-profit sectors.

\section*{Data availability}
No data were used for the research described in the article.

\end{document}